\documentclass{amsproc}

\usepackage{amssymb}

\newtheorem{thm}{Theorem}[section]
\newtheorem{lem}[thm]{Lemma}
\newtheorem{prop}[thm]{Proposition}
\newtheorem{cor}[thm]{Corollary}
\newtheorem{exs}[thm]{Examples}
\newtheorem{ex}[thm]{Example}
\theoremstyle{definition}
\newtheorem{defn}[thm]{Definition}

\theoremstyle{remark}

\numberwithin{equation}{section}

\newcommand{\sbq}{\subseteq}

\newcommand{\mc}{\mathcal}
\newcommand{\vide}{\emptyset}
\newcommand{\tbf}{\textbf}
\newcommand{\mbf}{\mathbf}
\newcommand{\mf}{\mathfrak}
\newcommand{\qcl}{\ensuremath{Q_{\text{cl}}\,}}
\newcommand{\inv}{^{-1}}

\newcommand{\e}{\mbf{e}}

\DeclareMathOperator{\spec}{Spec}
\DeclareMathOperator{\minspec}{Min}
\DeclareMathOperator{\ann}{ann}
\DeclareMathOperator{\annb}{ann_{\mbf{B}}}
\DeclareMathOperator{\coz}{coz}
\DeclareMathOperator{\z}{z}
\newcommand{\R}{\mathbf R}
\newcommand{\N}{\mathbf N}

\newcommand{\Z}{\mathbf Z}
\newcommand{\B}{\mathbf B}
\newcommand{\ze}{\text{z}}

\newcommand{\varep}{\varepsilon}
\newcommand{\res}{\raisebox{-.5ex}{$|$}}
\newcommand{\ov}{\overline}
\newcommand{\cx}{\text{C}(X)}
\newcommand{\cxs}{\text{C}^*(X)}
\newcommand{\rr}{\tbf{rr}}

\newcommand{\bd}{\text{\tbf{bd}\,}}
\newcommand{\inte}{\text{\tbf{int}}}
\newcommand{\rro}{\le_{\text{rr}}}
\newcommand{\wrr}{\wedge_{\text{rr}}}
\newcommand{\vrr}{\vee_{\text{rr}}}

\newcommand{\W}{\text{\tbf{W}}}

\begin{document}

 \title[Reduced ring order]{The reduced ring order and lower semi-lattices}

 \author{W.D. Burgess}
\address{Department of Mathematics and Statistics\\ University of Ottawa, Ottawa, Canada, K1N 6N5}
\curraddr{}
\email{wburgess@uottawa.ca}
\thanks{}

\author{R. Raphael}
\address{Department of Mathematics and Statistics\\ Concordia University, Montr\'eal, Canada, H4B 1R6}
\curraddr{}
\email{raphael@alcor.concordia.ca}
\thanks{The authors would like to thank the referee for a thorough reading and for some helpful suggestions.}

\dedicatory{\emph{Preprint Version}}
\date{}

\begin{abstract} Every reduced ring $R$ has a natural partial order defined by $a\le b$ if $a^2=ab$; it generalizes the natural order on a boolean ring. The article examines when $R$ is a lower semi-lattice in this order with examples drawn from weakly Baer rings (pp-rings) and rings of continuous functions. Locally connected spaces and basically disconnected spaces give rings $\cx$ which are such lower semi-lattices. Liftings of countable orthogonal (in this order) sets  over a surjective ring homomorphisms are studied. 
\end{abstract}

\maketitle

\setcounter{section}{1}\setcounter{thm}{0}
\noindent\tbf{1. Introduction.} The subject of this article is an aspect of  the \emph{reduced ring order} (here called the
 \rr-order). This partial order is defined in all reduced rings (rings with no non-trivial nilpotent elements). The global assumption in the sequel will be that \emph{the rings considered are unitary and reduced} and, for some parts, also commutative. If $R$ is a reduced ring the set of idempotents of $R$, which are here all central, is denoted $\B(R)$. A \emph{regular ring} will always mean \emph{von Neumann regular}.
\begin{defn}\label{defrr} Let $R$ be a reduced ring then for $a,b\in R$, $a\rro b$ if $ab=a^2$. \end{defn} 
This is readily seen to be a partial order which has been called variously the \emph{Abian order} and the \emph{Conrad order} (see, for example, \cite{BR} which dealt with orthogonal completions in this order).  In a reduced ring $R$, if $a^2=ab$ then, also, $a^2=ba$. This order generalizes to all reduced rings the natural order on boolean rings. 

In a reduced ring $R$ and $r\in R$ the left and right annihilators of $r$ coincide and are denoted $\ann r$. The following well-known property of reduced rings, obvious in the commutative case,   will be used often. It was first shown in \cite{AR} and \cite{St} and is found in e.g.~\cite[Lemma~12.6]{L}. 

\begin{lem}\label{c.prime} If $R$ is a reduced ring and $\mf{p}$ is a minimal prime ideal of $R$ then $\mf{p}$ is completely prime, i.e., $R/\mf{p}$ is a domain. \end{lem}

The \rr-order can be a very coarse one but there are important classes of rings where it plays an interesting role. To illustrate one extreme, if $R$ is a domain, then 0 is the unique minimal element in the order and no other pairs of elements are related. In fact in any reduced ring 0 is the unique minimal element while non zero-divisors are maximal. At the other extreme, a reduced ring $R$ is a lattice in the \rr-order only when $R$ is boolean (Proposition~\ref{bool}) but the main question asked below is when a reduced ring is a lower semi-lattice in this order  (\rr-lower semi-lattice). Here there are many rings where the answer is positive. Moreover, some classes of rings are characterized by this property, for example weakly Baer rings within almost weakly Baer rings (Theorem~\ref{awB}). It is shown that if $R$ is a commutative reduced ring it can be embedded, as an essential extension, in a weakly Baer ring; this extension is proper if $R$ is not already weakly Baer (Proposition~\ref{embedwB} and Corollary~\ref{W2=W}).

If $R\to S$ is a surjective homomorphism of reduced rings when can a countable orthogonal family (in the \rr-order) in $S$ be lifted to an orthogonal family in $R$? This is answered in two cases, Theorem~\ref{liftwb} and Proposition~\ref{liftcx}. 

Among rings of continuous functions, two quite different classes of topological spaces yield \rr-lower semi-lattices, locally connected spaces (Theorem~\ref{locc}) and basically disconnected spaces (\cite[Example~3.2]{NR}). A characterization of these spaces is available  (Theorem~\ref{finfg})  but requires looking at each pair of functions. However, even if $X$ is a compact, metric and connected space it does not follow that $\cx$ is an \rr-lower semi-lattice (Theorem~\ref{locc}(2)). 

Consider the case of $R=\prod_{\alpha\in A}D_\alpha$ where each $D_\alpha$ is a domain. Here, any pair of elements has an infimum, noted $\wrr$. Let $(a_\alpha)$ and $(b_\alpha)$ be elements of $R$. Define $(c_\alpha)\in R$ by $c_\alpha = a_\alpha$ when $a_\alpha = b_\alpha$ and $c_\alpha =0$ when $a_\alpha\ne b_\alpha$. It is readily seen that $(c_\alpha) =(a_\alpha)\wrr (b_\alpha)$. 

Such computations are not readily available in a general reduced ring $R$. The question can be posed in the following way. Embed $R$ in $S=\prod_{\mf{p}\in \minspec R}R/\mf{p}$, where $\minspec R$ is the minimal spectrum.  Given $a,b\in R$ then $c=a\wrr b$ can be found in $S$ but that element need not be in $R$. What is the ``best approximation'' to $c$ which can be found in $R$? This question will be considered in the sequel, sometimes directly, sometimes indirectly. See Corollary~\ref{orthog}, below for an easy case. It is also shown that if $|\minspec R| \le 3$ then $R$ is an \rr-lower semi-lattice but not necessarily for more than three minimal prime ideals (Proposition~\ref{Goldie}). \\[-.5ex]

\setcounter{section}{2}\setcounter{thm}{0}
\noindent\tbf{2. The reduced ring order and lower semi-lattices.} Since the \rr-order on a domain $R$ is very easy to understand and makes $R$ into a lower semi-lattice, the first place to look for examples is where rings are built out of domains in a straightforward way.

Recall (\cite[Chapter~V2]{J}) that a ring $R$ can be represented as the ring of sections of a sheaf, called the \emph{Pierce sheaf} over the boolean space $\spec \B(R)$, where $\B(R)$ is the boolean algebra of central idempotents of $R$. (In a reduced ring, all the idempotents are central.) The stalks, for $x\in \spec\B(R)$, are the rings $R_x = R/Rx$. For $a\in R$, put $a_x= a+Rx$ and define $\ze  (a)=\{x\in \spec\B(R)\mid a_x=0\}$, an open set in $\spec\B(R)$. Its complement is $\coz a$. 

An interesting case in this study is where the stalks of this representation are all domains. For a reduced ring $R$, what is wanted is that the $Rx$, as above, are precisely the minimal prime ideals of $R$. In the commutative case these rings were characterized in \cite[Theorem~2.2]{NR} as the rings $R$ in which the annihilator of an element is generated by idempotents; these rings are called \emph{almost weakly Baer rings}.   The name will here be abbreviated to \emph{awB rings}. For the purposes of this article, the definition will be extended to \emph{all reduced rings}. The awB rings include commutative (Von Neumann) regular rings and, more generally, strongly regular rings, and Baer rings, among others.  The characterization of rings whose Pierce stalks are domains extends readily to general reduced rings as shown in Proposition~\ref{NRgen}.

The next proposition gives some information about the \tbf{rr}-order.  If a reduced ring $R$ is viewed as a subring of $\prod_{\mf{p}\in \minspec R}R/\mf{p}$, a product of domains, and $r,s\in R$, to ask about $r\wrr s$ is to ask to when there is an element  $t\in R$ which is zero when $r+\mf{p}$ and $s+\mf{p}$ differ, coincides with them when $t+\mf{p}\ne 0$ and is the unique element \rr-maximal among all such elements.

\begin{prop} \label{lsl}  Let $R$ and $S$ be reduced rings. (i)~Let  $a, b,c\in R$ where $c$ commutes with $a$: if $a\rro b$ then $ca\rro cb$ and, in particular, if $e=e^2$ then $ae\rro be$. 

(ii)~If $a,b\in R$, $c= a\wrr b$ and $u$ is a central unit in $R$, then $uc = ua\wrr ub$. 

(iii)~If  $\phi \colon R\to S$ is a ring homomorphism then $\phi$ preserves the \tbf{rr}-order on $R$. 

(iv)~If $R$ and $S$ are \rr-lower semi-lattices and $\phi\colon R\to S$ a surjective homomorphism with kernel $K$, where $K$ is generated by idempotents, then $\phi$ preserves the lower semi-lattice structure. \end{prop} 

\begin{proof} Notice that if $a^2=ab$ then, also, $a^2=ba$ since $(ab-ba)^2=0$.

(i) Clearly $a^2= ab$ implies $(ca)^2 = cacb$. Moreover, idempotents are central in $R$.

(ii)  By (i), $uc\rro ua$ and $uc\rro ub$. Suppose $d\rro ua$ and $d\rro ub$, then $u\inv d\rro a$ and $u\inv d\rro b$. Hence, $u\inv d\rro c$ giving $d\rro uc$ and $uc= ua\wrr ub$. 

(iii) Obvious.

(iv) If $a, b\in R$ and $c=a\wrr b$ it must be shown that $\phi(c) = \phi(a)\wrr \phi(b)$. Put $s= \phi(a)\wrr \phi(b)$. Clearly $\phi(c)\rro s$, using (iii). Pick $t\in \phi\inv(s)$. There exist idempotents  $e, f\in K$ and $r, r'\in R$ so that $t^2-ta= er$ and $t^2-tb = fr'$. Let $g =(1- e)(1-f)$. It follows that $(t^2-ta)g =0 =(t^2-tb)g$ which implies $(tg)^2= tga$ and $tg\rro a$ and, similarly, $tg\rro b$. Thus, $tg\rro c$. Note that $\phi(g)=1$. Applying $\phi$ gives $\phi(tg)= \phi(t)= s\rro \phi(c)$. This shows that $\phi(c) = s$. \end{proof}

\begin{cor} \label{orthog} Let $R$ be a reduced ring and $\mf{P}(R) = \prod_{\mf{p}\in \minspec}R/\mf{p}$. If for $a,b\in R$, $a\wrr b=0$ in $\mf{P}(R)$ then $a\wrr b=0$ in R. In particular, if $ab=0$ then $a\wrr b=0$ in both rings. \end{cor} 

\begin{proof} This follows immediately from Proposition~\ref{lsl}~(iii) and how infs are calculated in $\mf{P}(R)$. \end{proof}

In Proposition~\ref{lsl}~(iv) the hypotheses on $\phi$ are required. If, for example, $\phi$ is a surjective homomorphism from a domain $R$ onto a domain $S$ with non-zero kernel and if $a \ne b$ are not in $\ker \phi$ with $\phi (a) = \phi (b)$,  then $a\wrr b=0$ but $\{a,b\}$  is sent to $\{\phi(a)=\phi(b)\}$ and $\phi(a)\wrr \phi(b)=\phi(a)\ne 0$.

In what follows the case of a not necessarily commutative reduced ring $R$ which is a Pierce sheaf of domains will be examined. To simplify notation, for $a \in R$, $\annb a = \ann a \cap \B(R)$.  The awB rings are those for which $\ann a = (\annb a)R$.  A \emph{weakly Baer ring} (wB ring) is a reduced ring such that for each $r\in R$ there is $e\in \annb r$ such that $\ann r = Re$. This means that for $x\in \spec \B(R)$, $r_x=0_x$ if and only if $e\notin x$. (Such rings are also called \emph{pp-rings}, for example in \cite{B}, or \emph{Rickart rings} in \cite{BPR}.) 

To simplify notation the following definition is included.

\begin{defn} \label{rr-lower}  A reduced ring $R$ which is an \rr-lower semi-lattice is called \emph{\rr-good}. If, for $r,s\in R$, $r\wrr s=0$, the elements $r$ and $s$ are called \rr-orthogonal.\end{defn}

The property of being \rr-orthogonal is what is called a \emph{Pierce property}.

\begin{lem}\label{Pprop} Let $R$ be a reduced ring. Then, for $r,s \in R$, $r\wrr s=0$ if and only if for all $x\in \spec \B(R)$, $r_x\wrr s_x =0_x$.    \end{lem}

\begin{proof}  Suppose first that $r\wrr s=0$. If, for some $x\in \spec \B(R)$, there is $a\in R$ with $0_x\ne a_x\rro r_x $ and $a_x\rro s_x$ in $R_x$. Then there is $e\in \spec \B(R)\setminus x$ such that $a^2e= aer=aes$ and $ae\ne 0$, a contradiction. Suppose now that $r_x$ and $s_x$ are \rr-orthogonal in each stalk. If there is $0\ne a$ with $a^2=ar = as$ then, for some $x\in \spec \B(R)$, $a_x\ne 0_x$ leading to a contradiction.  \end{proof}

Before going on to characterize the \rr-good rings among the awB rings it will be shown that the Niefield-Rosenthal characterization of commutative rings whose Pierce stalks are domains extends readily to all reduced rings.

\begin{prop}\label{NRgen}  Let $R$ be a reduced ring. Then each of the stalks of the Pierce sheaf of $R$ is a domain if and only if $R$ is an awB ring. \end{prop}

\begin{proof} Recall the notation for the Pierce sheaf: the base space is $\spec \B(R)$ and the stalks are the rings $R/Rx$, $x\in \spec\B(R)$. Suppose first that $R$ is awB. If for $r,s\in R$ and $x\in \spec \B(R)$, $r_x \ne 0_x$ and $(rs)_x=0_x$ then there is $e\in \B(R)\setminus x$ such that $rse=0$. Hence, by the awB property, there are idempotents $e_1,\ldots, e_k\in \ann r$ and $t_1,\ldots, t_k\in R$ such that $se = \sum_{i=1}^k t_ie_i$. However, the $e_i\in x$ showing that $se\in Rx$ and, thus, $(se)_x=s_x =0_x$ because $e_x=1_x$. Hence, $R_x$ is a domain.

In the other direction, if each $R_x$ is a domain then, for $r\in R$, if  $s\in \ann r$ then $s_x=0_x$ if $r_x\ne 0_x$. Hence, $\coz s \sbq \z(r)$, where $\ze(r)$ is an open set. Since $\z(r)$ is a union of clopen (closed and open) sets of the form $\coz e$, $e\in \B(R)$ and $er=0$.  The compactness of the closed set $\coz s$ implies that $\coz s$ is covered by a finite set of the $\coz e$. The union of this finite collection is also clopen, say  $\coz f$, $f\in \B(R)$ and $f\in \ann r$. From this, $s = sf$. Hence, $\ann r= (\annb r)R$.
   \end{proof}
  
\begin{thm} \label{awB} Let $R$ be an awB ring. Then $R$ is \rr-good  if and only if $R$ is a wB ring. \end{thm}

\begin{proof}
Suppose first that $R$ is a wB ring. Given non-zero $a,b\in R$ with $a\ne b$, let $eR = \ann (a-b)$, $e=e^2$. The claim is that $ae = be = a\wrr b$.  First, $ae\rro a$ and $ae\rro b$ since $ae\cdot a = a^2e=(ae)^2$ and $ae\cdot b = be\cdot b=(be)^2 =(ae)^2$. If $c\rro a$ and $c\rro b$ then $c(a-b)=0$ so that $c=ce$. Then, $cae=c^2e= c^2$ showing that $c\rro ae$ and, similarly, $c\rro be$. 

Now suppose that $R$ is an awB ring which is \rr-good. Consider $r\in R$. It must be shown that $\ann r =eR$ for some $e\in \B(R)$.  Set $a = r+1$ and $b=1$. By assumption there is $c= a\wrr  b$. Since $c^2=ca=cb$, $c\in \ann (a-b)$. Notice that since $c^2=cb=c1=c$, $c\in \annb(a-b)$.  Let $f\in \annb(a-b)$. Set $g=f-cf$. Now $ga = g(r+1) =g$ so that $g\rro a$ and $g\rro b$, showing that $g\rro c$. Hence, $g=gc$, but $gc=0$ giving $g=0$ and thus $f\rro c$.  Hence, $c$ is the unique maximal element of $\annb (a-b)$ showing that $\ann r = cR$.  \end{proof}

The calculations above show the following in an awB ring $R$: $r$ and $s$ in $R$ have a non-zero \rr-lower bound if and only if there is $0\ne e=e^2$ with $e(r-s)=0$ and $er\ne 0$.  

Constructions of wB rings will be discussed below. Among rings of continuous functions, $\cx$ is wB if and only if $X$ is basically disconnected (see the remark before Example~3.2 in \cite{NR}). According to \cite[Example~3.2]{NR}, the ring of continuous real valued functions $\text{C}(\beta\N\setminus \N)$ is awB but not wB.  Rings of continuous functions $\text{C}(X,\Z)$, where $X$ is a topological space and $\Z$ has the discrete topology, are easily seen all to be wB.  Notice that the property \rr-good is not a Pierce property. If $R$ is \rr-good then each Pierce stalk is \rr-good but the example $\text{C}(\beta \N\setminus \N)$ has stalks domains and, hence, \rr-good but the ring is not, by Theorem~\ref{awB}.

 Before pursuing consequences of Theorem~\ref{awB} it is worthwhile to observe how rarely pairs of elements, even in a wB ring, can have an \rr-supremum (denoted $\vrr$).  Moreover, it is shown below that $R$ is an \rr-upper semi-lattice if and only if it is boolean.

Assume for now that $R$ is an awB ring and that $r,s\in R$.  Consider the following pairwise disjoint subsets of $\spec \B(R)$:
\begin{equation*} \begin{array}{cccccc}
A_{r,s}&=& \coz r\cap \coz s\cap \z(r-s)\;,&
B_{r,s}&=& \coz r \cap \z(s)\\ C_{r,s}&=& \z(r)\cap \coz s\;,&
D_{r,s}&=& \z(r)  \cap \z(s)\end{array}
\end{equation*}

If, in addition $R$ is wB, each $\coz t$ is the cozero set of an idempotent.  For the given elements $r$ and $s$ the following idempotents, $e_1, e_2, e_3, e_4$ are defined by: $\coz e_1 = A_{r,s}$, $\coz e_2= B_{r,s}$, $\coz e_3 = C_{r,s}$ and $\coz e_4= D_{r,s}$. 
Note that these idempotents are orthogonal. This notation will be used in the following proposition.

\begin{prop} \label{uslwB}  Let $R$ be an awB ring. (i)~Given $r,s \in R$ then $r\vrr s$ exists only if ($*)$~$\coz r\cup \coz s\sbq A_{r,s}\cup B_{r,s} \cup C_{r,s}$. (ii)~If, in addition, $R$ is wB then $(*)$ implies that $r\vrr s$ exists. \end{prop} 

\begin{proof} Notice that condition ($*$) implies that $A_{r,s}\cup B_{r,s} \cup C_{r,s} \cup D_{r,s} = \spec \B(R)$. 

(i)~Suppose $(*)$ fails but that $r$ and $s$ have an \rr-upper bound $c$. If, for example, $x\in \coz r$ and $x\notin A_{r,s}\cup B_{r,s}\cup C_{r,s}$ then $r_x\ne 0$, $s_x\ne 0$ and $r_x\ne s_x$. But then $r_x^2= r_xc_x$ shows $r_x=c_x$ but $s_x^2=s_xc_x$ gives $s_x=c_x$, a contradiction. Hence, $r$ and $s$ do not have an upper bound. Similarly if $x\in \coz s$ and $x\notin A_{r,s}\cup B_{r,s}\cup C_{r,s}$. 

(ii)~Suppose that $\coz r \cup \coz s\sbq \coz (e_1+e_2+e_3)$. Define $c= (e_1+e_2)r+e_3s$. Then, $r^2-rc = r^2-r(e_1+e_2)r-re_3s =0$ since $re_3=0$ and $r(e_1+e_2) =r$, by  ($*$). Similarly $sc = s^2$ since $c$ is also $ (e_1+e_3)s+ e_2r$. Thus, $c$ is an upper bound of $r$ and $s$. For any $t\in R$, $c+te_4$ is also an upper bound but $t=0$ gives a smallest one of the form $c+te_4$. It will be shown that any upper bound $d$ has this form.  If $d$ is an upper bound and $x\in \coz e_1$, $r_x^2=r_xd_x$ giving $d_x=r_x$ and, similarly $d_x= s_x$. Thus $c$ and $d$ are equal over $\coz e_1$. Similarly for $\coz e_2$ and $\coz e_3$ showing that $c$ and $d$ can only differ over $D_{r,s}$. Hence, $c = r\vrr s$.  \end{proof}

\begin{prop}\label{bool} A ring $R$ is an \rr-upper semi-lattice if and only if it is a boolean ring. \end{prop}

\begin{proof} Each boolean ring is an upper semi-lattice in its natural order, which coincides with the \rr-order. On the other hand, suppose $R$ is any reduced ring which is an \rr-upper semi-lattice.  If $a\in R$, $a\ne 1$, put $c= a\vrr 1$. Then, $1^2= 1c$ so that $c=1$. Then $a^2= a$. \end{proof}

The next result looks at surjective homomorphisms from a wB ring.  The following lemma will be useful.

\begin{lem}\label{2elts}  Let $R$ be a wB ring and $a,b\in R$. There is a largest $e\in \B(R)$ such that $a\wrr eb=0$.  \end{lem}

\begin{proof} By Theorem~\ref{awB}, $R$ is \rr-good. Let $c= a\wrr b$ and $e\in \B(R)$ be such that $\ann c = eR$. Then for $x\in \spec \B(R)$, $c_x =0_x$ if and only if $e_x=1_x$. In other words, $c_x=a_x =b_x \ne 0_x$ if and only if $e_x=0_x$. Then there is no $x$ so that $a_x= (eb)_x \ne 0_x$. This shows that $a\wrr eb=0$. Now suppose that $f\in \B(R)$ is such that $e<f$. Thus, there is $x\in \spec \B(R)$ such that $f_x= 1_x$ while $e_x= 0_x$. Then, for each such $x$ it follows that $c_x=a_x=b_x  =(fb)_x\ne 0_x$ which shows that $a\wrr fb\ne 0$ 
by Lemma~\ref{Pprop}. \end{proof}

The following result about lifting \rr-orthogonal sets also gives another family of examples in the study of lifting orthogonal sets in \cite{R}. See also Proposition~\ref{liftcx}, below.

\begin{thm} \label{liftwb} Let $R$ and $S$ be reduced rings where $R$ is wB  and $\phi \colon R\to S$ is a surjective ring homomorphism such that $\phi (\B(R)) = \B(S)$. If $\{s_n\}_{n\in \N}$ is an \tbf{rr}-orthogonal set in $S$ then it lifts, via $\phi$, to an \rr-orthogonal set in $R$. \end{thm}

\begin{proof} Let $r_1\in R$ be any lifting of $s_1$.  Suppose for $n\ge 1$ that $s_1, \ldots, s_n$ have been lifted to $r_1, \ldots, r_n$, respectively, where $\{r_1,\ldots, r_n\}$ is \rr-orthogonal in $R$. Let $u$ be any lifting of $s_{n+1}$. Notice that for any pair $1\le i<j\le n$, $r_i$ and $r_j$ never coincide and are non-zero in a Pierce stalk. Assume $\{x\in \spec\B(R) \mid \text{for some}\; 1\le i\le n, (r_i)_x=u_x\ne 0_x\} \ne\vide$ for otherwise $r_{n+1}=u$ can be used. As in Lemma~\ref{2elts} there are maximal idempotents $e_1, \ldots, e_n$ such that $r_i\wrr e_iu=0$, and, by what has been observed, $\{1-e_1, \ldots,1- e_n\}$ are orthogonal, in the sense of idempotents. This is because it is not possible to have, for $i\ne j$, $(r_i)_x=u_x =(r_j)_x\ne 0_x$, which would happen if $(e_i)_x= (e_j)_x =0_x$, i.e., if $((1-e_i)(1-e_j))_x \ne 0_x$.  Put $\varep = 1-((1-e_1)+ \cdots +(1-e_n))$. Then, for $x\in \spec \B(R)$, $\varep_x=0_x$ exactly where one of the $(e_i)_x=0_x$.  Now put $r_{n+1} = \varep u$.  It first needs to be shown that $r_i\wrr \varep u =0$ for $i=1,\ldots, n$  or, by Lemma~\ref{Pprop}, that $(r_i\wrr \varep u)_x= 0_x$ for all $x\in \spec \B(R)$. If $(\varep u)_x \ne 0_x$ then $\varep_x =1_x$ and then $(e_i)_x=1_x$ for all $i$. This means that $(r_i\wrr \varep u)_x=(r_i\wrr e_iu)_x=0_x$.  If $(\varep u)_x=0_x$ then, of course, $(r_i\wrr \varep u)_x= 0_x$. By Lemma~\ref{Pprop}, each $r_i\wrr \varep u =0$.

It must now be shown that $\phi(r_{n+1}) = s_{n+1}$. The restriction of $\phi$ to $\B(R)$ is, by hypothesis, onto $\B(S)$; call the restriction $\beta$.  The maximal ideals of $\B(S)$ are in one-to-one correspondence with those of $\B(R)$ containing $\ker \beta$.  It is only necessary to look at those $x\in \spec \B(R)$ where $\ker \beta \sbq x$. Consider now $x\in \spec \B(R)$ with $\ker \beta\sbq x$ and $\beta (x) = y\in \spec\B(S)$. (i)~If $\varep_x=0_x$ then, for some $i$, $1\le i\le n$, $(e_i)_x=0_x$. This means that $(r_i\wrr u)_x\ne 0$ and then $(r_i)_x= u_x \ne 0_x$. Then, $\phi(r_i)_y = (s_i)_y =\phi(u)_y= (s_{n+1})_y =0_y$, since $s_i$ and $s_{n+1}$ are orthogonal. However, $\phi(\varep u)_y=0_y$ as well. (ii)~If $\varep_x=1_x$ then for all $i$, $1\le i\le n$, $(r_i)_x \ne u_x$ or $(r_i)_x = u_x= 0_x$. In any case $\phi(r_i)_y = (s_i)_y \ne \phi(u)_i = (s_{n+1})_y$ or $\phi(r_i)_y = \phi(u)_y = 0_y$. Here, $\phi(u)_y= \phi(\varep u)_y$. 
Hence, in all cases, $\phi(u)_y= \phi(\varep u)_y = (s_{n+1})_y$. In other words, $\phi(r_{n+1}) = s_{n+1}$. \end{proof}

Notice that in Theorem~\ref{liftwb} that the hypotheses do not guarantee that $\phi$ preserves the \tbf{rr}-order infima but that, in any case, countable \rr-orthogonal sets in $S$ lift to \rr-orthogonal sets in $R$.  No example is known to date of a surjective homomorphism of reduced rings where the lifting of countable \rr-orthogonal sets fails.

 A strongly regular ring (see e.g., \cite[page~199]{L}) $R$ is a wB ring where the stalks are division rings. In this case Proposition~\ref{lsl}~(iii) can be strengthened. If $r\in R$ then there is $r'\in R$ with $r^2r' =r$ and $(r')^2r = r'$. The idempotent $rr'$ will be denoted $\e(r)$. If $\phi \colon R\to S$ is a surjective ring homomorphism with kernel $K$ then $S$ is also strongly regular, $\phi(\e(r)) = \e(\phi(r))$ and $K$ is generated by idempotents. Moreover, if $r\in K$ then $\e(r)\in K$ and \emph{vice versa}. 

\begin{cor} \label{str-reg}  Let $R$ and $S$ be strongly regular rings equipped with the \tbf{rr}-order, and $\phi \colon R \to S$ a surjective ring homomorphism. Then, $\phi$ is an \rr-lower semi-lattice homomorphism. Moreover, a countable \rr-order orthogonal set in $S$ can be lifted to an \rr-order orthogonal set in $R$. \end{cor}

\begin{proof} Since the Pierce stalks of strongly regular rings are division rings, these rings are awB but, in addition, each principal ideal is generated by an idempotent so that they are, in fact, wB rings. Thus, $R$ and $S$ are \rr-good (Theorem~\ref{awB}) and $\phi$ preserves the lower semi-lattice structure since all ideals are generated by idempotents (Proposition~\ref{lsl}~(iv)). The lifting follows from Theorem~\ref{liftwb}. \end{proof}

Notice that when $\phi \colon R\to S$ is a ring surjection of wB rings, $\phi$ need not be a \rr-lower semi-lattice morphism. Consider $\phi \colon \Z\to \Z/6\Z \cong \Z/2\Z \times \Z/3\Z$.  The \rr-orthogonal set $\{7,13\}$ goes to the pair $((1,1), (1,1))$, which is not \rr-orthogonal.

Theorem~\ref{awB} gives a class of reduced rings which are \rr-good and shows that there are reduced rings which are not. The following example is a reduced ring which is not \rr-good but which has a wB subring $S$ with $\B(R) = \B(S)$.  Section~3 presents examples of rings which are \rr-good but which are not awB, for example $\text{C}(\R)$, which has no non-trivial idempotents. Example~\ref{notlattice} will show up again in Example~\ref{eq-bad2}

\begin{ex}\label{notlattice}  Fix a prime $p$ and let $R$ be the ring of sequences of integers which are eventually constant modulo $p$.  Then, $R$ lies between the \rr-good rings $S$ of eventually constant integer sequences and $P= \prod_\N\Z$ with $\B(R) = \B(S)$. However, $R$ is not \rr-good.  \end{ex}

\begin{proof} The rings $S$ and $P$ are even wB. Notice that an idempotent in $R$ has finite or cofinite support. Consider $r$ which is constantly $1$ and $s = (1, p+1, 1, p+1, 1,\ldots)$. Thsn, $r$ and $s$ have non-zero \rr-lower bounds, for example $(1, 0,0,0,\ldots )$ or $(1,0, 1, 0,0,\ldots)$ but no largest one. \end{proof}

In \cite[Corollary~1.6]{NR} it is pointed out that in a commutative noetherian ring, the conditions awB and wB coincide.   A commutative noetherian reduced ring $R$ is an order in a finite product of fields and it is natural to ask if such a ring is \rr-good. When $R$ has three or fewer minimal prime ideals the answer is ``yes'' (neither commutativity nor the noetherian property is needed) but not in general. The proof for three minimal primes will show what can go wrong with four or more.

\begin{prop} \label{Goldie}  Let  $R$ be a reduced ring. 

(i)~Suppose $R$ has three or fewer minimal primes. Then $R$ is \rr-good.

(ii)~If $R$ has four minimal prime ideals it need not be \rr-good. \end{prop} 

\begin{proof} (i) The case of three minimal primes, $\mf{p}_1, \mf{p}_2, \mf{p}_3$, will be considered. Embed $R$ in $R/\mf{p}_1\times R/\mf{p}_2\times R/\mf{p}_3 =S$ via $\phi\colon R\to S$. The ring $R$ will be viewed as a subring of $S$. Consider $a,b\in R$ and their images in $S$, say $(a_1,a_2,a_3)$ and $(b_1,b_2,b_3)$.  If $a_i\ne b_i$ for $i=1,2,3$ then $a$ and $b$ have no non-zero lower bounds in $S$ and so,  $a\wrr b=0$ in $R$. If $a$ and $b$ differ in precisely two places, say, $a_2\ne b_2$ and $a_3\ne b_3$ then the inf in $S$ is $(a_1, 0,0)$.  If $(a_1,0,0,)\in R$ then that is the inf in $R$ otherwise $a\wrr b=0$. Finally suppose $a$ and $b$ coincide in the first two places but not in the third. The inf in $S$ is $(a_1,a_2,0)$. If it is in $R$ then it is $a\wrr b$. If not, there are three possibilities: $(a_1,0,0)\in R$, $(0,a_2,0) \in R$ and neither is in $R$.  In each case $a\wrr b$ exists in $R$. 

(ii) Let $K$ be a field. Let $Q$ be a product of four copies of the field $K(x,y,z)$. The ring $R$ is generated as a $K[x,y,z]$-subalgebra of $Q$ by the following elements:
\begin{equation*}\begin{array}{ccccccc}  r&=&(x^2+x,y,y,y) &\;\;&b_1&=&(0,z,0,0)\\ s&=&(x,y,y,y)&\;\;&b_2&=&(0,0,z,0)\\ a&=&(0,y,y,0)&\;\;&b_3&=&(0,0,0,z)\\ b&=&(0,0,y,y)&\;\;&\mf{1}&=&(1,1,1,1) \end{array}\end{equation*} The ring $R$ is noetherian with four minimal prime ideals and will be seen not to be \rr-good.

Note first that $\qcl(R) = Q$ since $r-s, b_1,b_2, b_3$ give the four idempotents of $Q$ in $\qcl(R)$.  The ring $R$ is noetherian with 4 minimal primes by Goldie's Theorem. 

Note that, in $R$, $a$ and $b$ are common lower bounds of $r$ and $s$, but $a$ and $b$ are not \rr-related. The only candidate for a lower bound of $r$ and $s$ which is larger than both $a$ and $b$ is $(0,y,y,y)$. The aim is to show that $(0,y,y,y)$  is not in $R$. Suppose on the contrary that it is. Then, there is a polynomial expression $F$, with coefficients in $K[x,y,z]$, in 7 variables with zero constant term which is evaluated at the vector $\mbf{v}=(r,s,a,b, b_1,b_2, b_3)$ to yield $F(\mbf{v}) +C = (0,y,y,y)$, where $C$ is a constant from $K[x,y,z]$.  

Since $z$ does not appear in an essential way in the main equation and is only used to obtain idempotents in $\qcl(R)$ it is possible to reduce the problem by setting $z=0$. Once this has been done there is a polynomial $f$ in four variables $u_1,u_2, u_3, u_4$ with coefficients in $K[x,y]$ and zero constant term. Set $\mbf{w}=(r,s,a,b)$ and then $f(\mbf{w}) +c = (0,y,y,y)$ for some constant $c\in K[x,y]$. The aim is to show that such an equation leads to a contradiction.  Subscripts will indicate the four components.

The polynomial $f$ can be written $f = h+f_a+ f_b+f_{ab}$ where $h$ has monomials only in $u_1$ and/or $u_2$, $f_a$ has monomials with factor $u_3$ but not $u_4$, $f_b$ has monomials with factor $u_4$ but not $u_3$, and $f_{ab}$ has monomials with factors both $u_3$ and $u_4$. Recall that all these polynomials have 0 constant term.

Note first that only $h$ influences the first component. Set $h(\mbf{w})_1= l$. Then, $c=(-l,-l,-l,-l)$. Moreover, $l$ is divisible by $x$. Now consider the other components. 

Notice the following: $f_a(\mbf{w})_1= f_a(\mbf{w})_4=0$, $f_b(\mbf{w})_1= f_b(\mbf{w})_2=0$, $f_{ab}(\mbf{w})_1= f_{ab}(\mbf{w})_2= f_{ab}(\mbf{w})_4=0$, and
\begin{equation*}\begin{array}{cccccc} (i)\;& h(\mbf{w})_2&=& h(\mbf{w})_3 &=&h(\mbf{w})_4\\(ii)\;&f_a(\mbf{w})_2&=& f_a(\mbf{w})_3& &\\
(iii)\;&f_b(\mbf{w})_3&=&f_b(\mbf{w})_4& &
\end{array}\end{equation*}
The elements $h(\mbf{w})_2,f_a(\mbf{w})_2, f_b(\mbf{w})_4$ are divisible by $y$ and  $f_{ab}(\mbf{w})_3$ is divisible by $y^2$. The components can now be computed.
\begin{equation*} \begin{array}{ccccc} (f(\mbf{w})+c)_2&=& h(\mbf{w})_2 +f_a(\mbf{w})_2 -l &=&y\\ (f(\mbf{w})+c)_4&=& h(\mbf{w})_4+ f_b(\mbf{w})_4 -l&=& y\\ (f(\mbf{w})+c)_3&=& h(\mbf{w})_3+ f_a(\mbf{w})_3+ f_b(\mbf{w})_3 +f_{ab}(\mbf{w})_3 -l&=& y\end{array} \end{equation*}
It follows that $h(\mbf{w})_2+f_a(\mbf{w})_2 = l+y$ and $h( \mbf{w})_4+f_b(\mbf{w})_4 = l+y$ and, using (i), $f_b(\mbf{w})_4 = f_a(\mbf{w})_2$. In the expression for $(f(\mbf{w})+c)_3$ the sum of the first two terms equals $l+y$ showing that $f_b(\mbf{w})_3 +f_{ab}(\mbf{w})_3 =0$. Since $f_{ab}(\mbf{w})_3$ is divisible by $y^2$, so is $f_b(\mbf{w})_3$ and, from that, so is $f_a(\mbf{w})_3$. This last is because $f(\mbf{w})_2= f(\mbf{w})_4$ and the various equalities with (ii) and (iii) yield $f_a(\mbf{w})_3 = f_b(\mbf{w})_3$.

In ($*$)~$h(\mbf{w})_3+ f_a(\mbf{w})_3+f_b(\mbf{w})_3+ f_{ab}(\mbf{w})_3 = l+y$ all the terms on the left are divisible by $y$. Hence, $l$ is divisible by $y$. To proceed, a more detailed look at $h(u_1,u_2)$ is needed. Let $A$ be a finite set of pairs of non-negative integers $i,j$ with $i+j\ge 1$ and, for $i,j\in A$, $a_{ij}\in K[x,y]$.  Each $a_{ij}$ can be written $a_{ij}'+a_{ij}''$, where $a_{ij}'$ contains all the terms of $a_{ij}$ with factor $y$ and $a_{ij}''$ all the remaining terms.   Write
\begin{eqnarray*}(**)&l=h(x^2+x,x)&=\sum_Aa_{ij}(x^2+x)^ix^j \\&&=\sum_Aa_{ij}'(x^2+x)^ix^j +\sum_Aa_{ij}'' (x^2+x)^ix^j\;.\end{eqnarray*} Since $y|l$, the last term in ($**$) is zero. 

From the expression for $f(\mbf{w})_3$, it follows that $f_a(\mbf{w})_3+ f_b(\mbf{w})_3+ f_{ab}(\mbf{w})_3= l+y - h(\mbf{w})_3$ is divisible by $y^2$. Since $x|l$ and $y|l$, one can write $l =l'xy$. Now, $$h(\mbf{w})_3 = \sum_Aa_{ij}'y^{i+j}+ \sum_{i+j\ge 2}a_{ij}''y^{i+j} +(a_{10}''+a_{01}'')y\;.$$ The first two terms on the right of this expression are divisible by $y^2$ and, hence, $y^2$ divides $l'xy +y -(a_{10}''+ a_{01}'')y$ and $y$ divides $l'x+1 -(a_{10}''+a_{01}'')$. It follows that $a_{10}''+a_{01}''= px+1$ where $p\in K[x]$. Now return to ($**$) to look at the last terms which are $a_{10}''(x^2+x) +a_{01}''x = a_{10}''x^2+ (px+1)x$. There are no other degree 1 terms in $x$. This contradicts the fact that $\sum_Aa_{ij}''(x^2+x)^ix^j =0$ and concludes the proof. 
 \end{proof}
 
 The remainder of this section is devoted to showing how to embed each commutative reduced ring into a wB ring in various ways.  Every such ring will be seen to have an essential wB extension by using the complete ring of quotients in the next proposition. It does not require commutativity.
 
 \begin{prop}\label{embedwB}  Let $R\sbq S$ be  reduced rings where $S$ is wB. Let $T\sbq S$ be the ring generated by $R$ and those idempotents from $\B(S)$ which generate the annihilators, in $S$, of the elements of $R$. Then, $T$ is wB.  \end{prop}

\begin{proof} The following notation will be used. For $r\in R$, let $\ann_Sr = e(r)S$, where $e(r) \in \B(S)$. Set $\mc{S} = \{e(r)\mid r\in R\}$ and $T = R[\mc{S}]$. It needs to be shown that the annihilator in $T$ of an element of $T$ is generated by an idempotent.  A typical non-zero element $t\in T$ can be written $t=\sum_{i=1}^kr_ie_i$ where the $0\ne r_i\in R$ and the $e_i$ are idempotents which, after the standard manipulations on idempotents, can be assumed to be orthogonal and derived from elements of $\mc{S}$ (and are thus in $T$). Put $\varep = 1-(e_1+\cdots +e_k)$. The claim is that $\ann_T t$ is generated by $e=\sum_{i=1}^ke(r_i)e_i +\varep$. It is clear that $te=0$. Let $a\in \ann_Tt$. Then $a$ can be written in the form $a=\sum_{j=1}^ms_jf_j$, with each $s_j\in R$ and the $f_j$ a set of orthogonal idempotents derived from those in $\mc{S}$. For each $i=1,\ldots, k$, the orthogonality of the $e_i$ shows that $0=r_ie_i\sum_{j=1}^ms_jf_j$ so that for each $j=1,\ldots, m$, $s_jf_je_i\in e(r_i)T$ and $s_jf_je_i\in e_ie(r_i)T$. From this, $0 =(\sum_{i=1}^kr_ie_i)(\sum_{j=1}^ms_jf_j) = \sum_{j=1}^m\sum_{i=1}^kr_is_je_if_j= \sum_{i=1}^mr_ie_ia$. By the orthogonality of the $e_i$, each $r_ie_ia=0$, which shows that $e_ia=e_iae(r_i)$. From this, $a = a(\sum_{i=1}^ke_ie(r_i)) +a\varep \in (\sum_{i=1}^ke_ie(r_i) + \varep) T$, as required.\end{proof}

Before applying this proposition it should be noted that neither the category of commutative wB rings, $\mc{WB}$, nor the category of commutative \rr-good rings, $\mc{RG}$, is a reflective subcategory of the category of commutative reduced rings $\mc{CR}$.  (See \cite[Chapter~V]{AHS}.) Both $\mc{WB}$ and $\mc{RG}$ are closed under products but neither is closed under equalizers as would be required for a reflective subcategory.  Hence, in both cases, the inclusion of it into $\mc{CR}$ cannot have a left adjoint. This is seen using the following examples. The first simple example, pointed out by J. Kennison, illustrates this for $\mc{WB}$.

\begin{ex}\label{eq-bad}  Let $R = \Z\times \Z$ and $S=Z_2$ with $\phi\colon R\to S$ defined by $\phi(a,b) = \bar{a}$ and $\psi\colon R\to S$ by $\psi(a,b) = \bar{b}$.  Then while $R$ and $S$ are wB, the equalizer $E$ of $\phi$ and $\psi$ is not  wB. \end{ex}

\begin{proof} Note that $(0,2)\in E$ but its annihilator is not generated by an idempotent in $E$, which has no non-trivial idempotents. \end{proof}

The same remark applies to the category $\mc{RG}$ in $\mc{CR}$. Here the example is a bit more complex.

\begin{ex}\label{eq-bad2}  Let $R=\prod_\N\Z$ and $S=\Z_2$. Let $\mc{U}$ be the set of all free ultrafilters on $\N$. For each $U\in \mc{U}$, let the ideal $I_U$ of $\prod_\N \Z_2$ be the set of those elements 0 over a set in $U$.    Define $\phi_U\colon R\to S$ by $(a_n)\mapsto (\ov{a_n}) \mapsto (\ov{a_n})+I_U$. Then the equalizer of all the $\phi_U$  is not \rr-good.  \end{ex}

\begin{proof} The ring in Example~\ref{notlattice} with $p=2$, which is not \rr-good,  will used. (Any prime $p$ could be used.) Suppose $r=(a_n)\in R$ is eventually constant modulo 2. Then, since the ultrafilters are free, $r$ will be in the equalizer. On the other hand, if $r=(a_n)\in R$ is not eventually constant modulo 2 the set, $X$, of components where $r$ is even and the set, $Y$, where they are odd are disjoint infinite sets. The sets $X$ and $Y$ are not both in all $U\in \mc{U}$.  Thus, $r$ will not have the same image under all of the $\phi_U$.  \end{proof}

In the case of a commutative reduced ring  $R$ (i.e., commutative semiprime rings) there are various standard ways of embedding $R$ into a regular ring. Each will yield a wB ring according to Proposition~\ref{embedwB}. The first, which is mentioned in passing, is the \emph{universal regular ring} studied in \cite{H} and \cite{Wi}. The functor $T\colon \mc{CR} \to \mc{VNR}$, where $\mc{VNR}$ is the category of commutative regular rings, embeds each $R\in\mc{CR}$ in  $T(R)\in \mc{VNR}$ so that the $T$ is the left adjoint of the inclusion functor $\mc{VNR}\to \mc{CR}$. If Proposition~\ref{embedwB} is used, it can be shown that the construction is functorial, say $U\colon \mc{CR} \to \mc{WB}$. However, $U$ does not preserve rings which are already wB and, hence, details will be omitted.  

Another choice of regular ring would be the complete ring of quotients, $Q(R)$ of $R\in \mc{CR}$, always regular in this case. This works well but is not functorial. There is another important regular ring between $R$ and $Q(R)$, namely the \emph{epimorphic hull}, the intersection of all regular rings between $R$ and $Q(R)$ (see \cite[Satz~11.3]{S}). However, the next lemma will show that the construction of Proposition~\ref{embedwB} gives the same result for any regular ring $S$, $R\sbq S \sbq Q(R)$. 

\begin{defn}\label{W(R)}  Let $R\in \mc{CR}$ and $Q(R)$ its complete ring of quotients. The wB ring constructed from $R$ and $Q(R)$ using Proposition~\ref{embedwB} is denoted $\mbf{W}(R)$. \end{defn}

It will be seen that an annihilator which is generated by an idempotent behaves well when passing from a ring to an essential extension ring. If $R\sbq S$ are reduced rings then the extension is called \emph{(right) essential} if every non-zero right ideal of $S$ has non-zero intersection with $R$. A right essential extension is also called a \emph{right intrinsic extension} in \cite{FU}. (Note that this is weaker than saying that $R_R$ is essential in $S_R$.)

\begin{lem}\label{essex}  Let $S\sbq T$ be reduced rings where $T$ is a right essential extension of $S$. Suppose that, for $s\in S$, there are idempotents $e\in S$ and $f\in T$ such that $\ann_S s =eS$ and $\ann_T s=fT$. Then, $e=f$. \end{lem}

\begin{proof} Since $es=0$, $e=ef$. If $f\ne ef$ then there is $t\in T$ such that $0\ne (f-ef)t\in S$. However, $s(f-ef)t =sf(1-e)t =0$ so that $f(1-e)t=f(1-e)te =0$, a contradiction. Hence, $e=f$. \end{proof}

\begin{cor}\label{uniquee}  Let $R\in \mc{CR}$ and $S\in \mc{CR}$ an essential extension of $R$ which is wB.  Let $V$ be the wB ring generated in $S$ as in Proposition~\ref{embedwB}. If $r\in R$ is such that $\ann_Rr=eR$ for some idempotent $e\in R$ then $e
S= \ann_Sr$. Thus, in the construction of $V$, the idempotent for $r\in R$ is already in $R$. \end{cor}

\begin{proof} This is a direct application of Lemma~\ref{essex}. \end{proof}

Everything is now in place to show that the construction of $\W(R)$ not only gives an efficient way of constructing a wB ring from $R$ but also that if $R$ is already wB then $\W(R) =R$. 

\begin{cor}\label{W2=W}  If $R\in \mc{CR}$ is a wB ring then $\W(R) =R$. \end{cor}

\begin{prop} \label{uniqueW}  Let $R$, $S$ and $V$ be as in Corollary~\ref{uniquee}. Then, $V \cong \W(R)$.\end{prop} 

\begin{proof} There is a copy of $Q(R)$ in $Q(S)$ (\cite[Satz~10.1]{S}). The wB ring $\W(R)$ can be calculated inside this copy of $Q(R)$. Now consider $r\in R$ and the idempotents $e$ and $f$ where $eV = \ann_Vr$ and $fQ(R) = \ann_{Q(R)} r$.  Then, $\ann_{Q(S)}r =eQ(S)$ and Lemma~\ref{essex} shows that $e=f$. Hence, $V$ and $\W(R)$ are generated over $R$ by the same set of idempotents. \end{proof}

The construction of $\W(R)$ can be seen, using Lemma~\ref{essex}, to be an example of a wB \emph{absolute to $Q(R)$ (right) ring hull} of $R$ for any  commutative reduced ring, in the language of  \cite[Definition~8.2.1(i)]{BPR}. In the case of a general reduced ring $R$, the construction of Proposition~\ref{embedwB} along with Lemma~\ref{essex} will yield a wB right ring hull of $R$ (\cite[Definition~8.2.1(iii)]{BPR})  \emph{if there is any} essential right wB extension $T$ of $R$.  

The process of constructing $\W(R)$ is illustrated by the following. 

\begin{prop} \label{qclreg}  Suppose $R\in \mc{CR}$ is such that its classical ring of quotients $Q=\qcl(R)$ is regular. Then $\W(R)$ is the ring generated over $R$ by $\B(Q)$.  \end{prop}

\begin{proof} In this case for $r\in R$, $\ann_Qr = (\ann_Rr ) Q$ and for each $a/b\in Q$, $\ann_Qa/b= (\ann_Ra)Q$. Hence, in Proposition~\ref{embedwB}, all the idempotents of $Q$ are adjoined. \end{proof}

\setcounter{section}{3}\setcounter{thm}{0}
\noindent\tbf{3. Rings of continuous functions.} In this section examples of \rr-good rings will be found which are quite unlike wB rings. The basic reference is \cite{GJ}.

Suppose $X$ is a completely regular space and $\cx$ is considered with the \tbf{rr}-order.  What does it mean for $\cx$ to be \rr-good?  Let $f,g\in \cx$ and suppose $h=f\wrr g$ exists. Then where $h$ is non-zero it must coincide with $f$ and $g$.  This means that $\coz h \sbq \z (f-g)\cap \coz f$. (Here $\coz f$ and $\z(f)$ have their usual meanings in rings of functions.) Moreover, on $\coz h$, $h$ coincides with $f$ and with $g$. These conditions only make $h \rro f$ and $h\rro g$. To have the infimum there must be a \emph{unique} largest cozero set $U\sbq \z (f-g) \cap \coz f$ and $h$ with (i)~$\coz h =U$, (ii)~$h$ and $f$ (and $g$) coincide on $U$, and  (iii)~for every net $\Lambda =\{\lambda_\alpha\}$ in $U$ converging to a point $y$ on the boundary of $U$, $\{f(\lambda_\alpha)\} \to f(y) =0$ and $\{g(\lambda_\alpha)\} \to g(y) =0$. Notation: For $Y\sbq X$ the boundary of $Y$ is denoted $\bd Y =\ov{Y}\setminus Y$.

\begin{defn} \label{xgood}  A completely regular topological space $X$ is called \rr-\emph{good} if the ring $\cx$ is \rr-good.\end{defn}

 As an example consider a pair of functions from $\text{C}(\R)$. (It turns out that $\R$ is an \rr-good space.) Let $f(x) = \sin x$ and $g(x) = |\sin x|$. The function $h(x)$ which is $f(x)$ where the two functions coincide and zero where they differ is continuous. Hence, $h = f\wrr g$. 
 
The purpose here is to find examples of \rr-good spaces.  It will be seen later that it is easy to construct spaces which are not \rr-good. Before looking at details it is worthwhile to note a class of \tbf{rr}-good spaces; more will follow. Recall that a space $X$ is \emph{basically disconnected} if the closure of a cozero set is open. As already mentioned these are precisely the spaces where $\cx$ is wB and hence \tbf{rr}-good by Theorem~\ref{awB}. These spaces include the P-spaces (\cite[4J]{GJ}) where $\cx$ is von Neumann regular. Note that a space can be totally disconnected without being \tbf{rr}-good: the space $X=\beta\N\setminus \N$ is an example where $\cx$ is awB but not wB (\cite[Example~3.2]{NR}) and, by Theorem~\ref{awB} again, $X$ is not \tbf{rr}-good. 

A complete characterization of all \rr-good spaces as a class of well-known spaces is elusive since examples include basically disconnected spaces and, as will be seen, locally connected spaces. However, Theorem~\ref{finfg} does characterize \rr-good spaces although by properties not easily translated into global language.
It makes more explicit what it means for $f$ and $g$ in $\cx$ to have a non-zero infimum. The description will later be used to get sufficient conditions for $X$ to be \tbf{rr}-good.

\begin{thm} \label{finfg}  Let $X$ be a completely regular topological space. Fix $f,g\in\cx$, $f \ne g$, and suppose that $f$ and $g$ have a non-zero \rr-lower bound. Let $\{h_\alpha\}_{\alpha\in A}$ be the set of non-zero \tbf{rr}-lower bounds of $f$ and $g$, $U_\alpha = \coz h_\alpha$, $\alpha\in A$ and $U=\bigcup_{\alpha\in A}U_\alpha$

(i)~Then, $f\wrr g$ exists if and only if $U$ has the property: ($*$)~if $x\in \bd U$ and $N$ is any neighbourhood of $x$ then, for some $\beta \in A$, $N\cap \bd U_\beta \ne \vide$. 

(ii)~Moreover, $f\wrr g$ exists if and only if $U$ has the property: ($**$)~if $x\in \bd U$ then for each $\varep > 0$ there is a neighbourhood $N_\varep$ of $x$ such that for all $\alpha \in A$, $N_\varep \cap U_\alpha \sbq \{y\in X\mid |h_\alpha(y)| < \varep\}$.  \end{thm}

\begin{proof}  (i)~Note, with $f$ and $g$ as above, that if $\bd U=\vide$ then $U$ is clopen and there is $e^2=e \in \cx$ such that $\coz e =U$. In this case $f\wrr g = ef =eg$ and $(*)$ is vacuously satisfied. 

Notice that, as in the statement, $h_\alpha$ coincides with $f$ and with $g$ on $\coz h_\alpha$; hence, $\coz h_\alpha\sbq \z(f-g) \cap \coz f$. Moreover, if $\alpha \ne \beta$ in $A$, then $(\bd U_\alpha) \cap U_\beta = \vide$. This is because for $y\in \bd U_\alpha$, $h_\alpha(y) =f(y) =0$ while for $x\in U_\beta$, $f(x) = h_\beta (x) \ne 0$. 

First assume property ($*$) for $f$ and $g$ which have a non-zero lower bound. The proof will show that, among the lower bounds, there is a maximal one.   With the $h_\alpha$ as in the statement, define $h$ by $h(x) = f(x)$ for $x\in U$ and $h(x) =0$ for $x\in X\setminus U$. The continuity of $h$ on $\bd U$ must be shown. The condition ($*$) says that for every neighbourhood $N$ of  $x\in \bd U$ there is some $\beta\in A$ with $N\cap \bd U_\beta \ne \vide$. Thus there is $y\in N$ with $f(y) =0$. By continuity of $f$ it follows that $f(x)=0$; the fact that $h$ is zero on $X\setminus U$ shows that $h$ is continuous at $x$ and that $h$ coincides with $f$ (and with $g$) on $U$. By construction $h\rro f$ and $h\rro g$ and, also, for each $\alpha \in A$, $h_\alpha \rro h$. 

In the other direction, if $h= f\wrr g$ exists it is one of the $h_\alpha$ showing that $\coz h = U$. By its continuity, the property ($*$) follows automatically. 

The proof of part (ii) is along the same lines as that of (i). \end{proof}

The next result gives a sufficient condition on a space $X$ to be an \rr-good space. For convenience, the condition will be given a name.

\begin{defn}\label{b-cond}  Let $X$ be a completely regular space. Let $\{U_\alpha\}_{\alpha \in A}$ be any family of non-empty cozero sets in $X$ with the following property: for $\alpha \ne \beta$ in $A$, $(\bd U_\alpha) \cap U_\beta = \vide$. The space $X$ is said to satisfy the \emph{B-property} (for \emph{boundary property}) if the following holds for each such family of cozero sets. Let $z\in \bd(\bigcup_{\alpha\in A}U_\alpha)$ then for every neighbourhood $N$ of $z$ there is $\beta \in A$ such that $N\cap \bd U_\beta \ne \vide$. \end{defn}

\begin{cor} \label{rrgood} Let $X$ be a completely regular topological space which satisfies the B-property.   Then $X$ is an \rr-good space.  \end{cor}

\begin{proof} The B-property is a stronger form (strictly stronger by Example~\ref{notB}, below) of the property~($*$) in Theorem~\ref{finfg}. This guarantees the existance of $f\wrr g$ whenever $f$ and $g$ have a non-zero lower bound. Otherwise, $f\wrr g =0$.\end{proof}

The next task is to find examples of spaces satisfying the conditions of Corollary~\ref{rrgood}.  Recall the definition of a locally connected space (e.g., \cite[page~374]{E}): $X$ is \emph{locally connected} if for each $x\in X$ and a neighbourhood $N$ of $x$ there is a connected subset $C\sbq N$ such that $x\in \inte (C)$. 

\begin{thm} \label{locc}  Let $X$ be a completely regular space. (1)~If $X$ is locally connected then $X$ is an \rr-good space. (2)~A connected space $X$ need not be \rr-good even when $X$ is a compact metric space.  \end{thm}

\begin{proof}  (1)~It will be shown that $X$ satisfies the B-property. Suppose $\{U_\alpha\}_{\alpha \in A}$ is a family of non-empty open sets in $X$ as in  the first part of Definition~\ref{b-cond}. Let $z\in \bd(\bigcup_{\alpha\in A}U_\alpha)$. It may be supposed that $z\notin \bd U_\beta$ for any $\beta \in A$. Choose a neighbourhood $N$ of $z$.  By \cite[6.3.3]{E}, there is a connected subset $C$ of $N$ with $z\in \inte C$. Then for some $\beta \in A$, $\inte C\cap U_\beta \ne \vide$. If $C\cap \bd U_\beta \ne \vide$ then there is nothing to prove. Otherwise, $C\cap \ov{U_\beta} = C\cap U_\beta$. This means that $C\cap U_\beta$ is clopen in $C$.  Moreover, $C\cap U_\beta \ne \vide$ and $C\cap U_\beta \ne C$ since $z\notin U_\beta$. This contradicts the fact that $C$ is connected. Hence, $N\cap \bd U_\beta \ne \vide$.

(2)~For the construction, \cite[Example~27.8(b)]{W} is used but truncated to make it compact. Define $X\sbq \R^2$ as follows. It consists of the vertical line segments $\{(0,y)\mid -1\le y\le 1\}$ and $\{(1,y)\mid -1\le y\le 1\}$ along with horizontal line segments $\{(x,1/n)\mid 0\le x\le 1\}$, for $n\in \Z\setminus \{0\}$, and $\{(x,0)\mid 0\le x\le 1\}$.  Elements $f,g\in \cx$ will be defined so that they do not have an \rr-inf.

\begin{equation*} f(x,y) = \begin{cases}2x & \text{if}\;0\le x\le 1/2, y=0 \;\text{or}\; y=1/n, n\in \Z\setminus\{0\}  \\ -2x+2&  \text{if}\; 1/2\le x\le 1, y=0 \;\text{or}\; y=1/n, n\in \Z\setminus \{0\}\\ 0& \text{otherwise} \end{cases} \end{equation*}
 
\begin{equation*}g(x,y) = \begin{cases}f(x,y) &  \text{if}\, x=0, x=1, \,\text{or}\,y=0 \;\text{or}\, y=1/n,\,n\in \Z\setminus \{0\}, n\;\text{even}\\ \frac{2|n|}{|n|+1}x &\text{if}\; 0\le x\le 1/2, y=1/n, n\in \Z\;\text{odd}\\ \frac{-2|n|}{|n|+1}x+ \frac{2|n|}{|n|+1}& \text{if}\; 1/2\le x\le 1, y=1/n, n\in \Z \;\text{odd} \end{cases} \end{equation*}
A typical \rr-lower bound of $f$ and $g$ coincides with $f$ and $g$ on finitely many of the horizontal lines, say with $y\in \{1/n_1,\ldots, 1/n_k\}$, the $n_i\ne 0$ and even, and zero elsewhere. However, continuity on the line where $y=0$ precludes a greatest \rr-lower bound. 
\end{proof}

The next result shows how to construct examples of spaces which are not \rr-good.

\begin{prop} \label{notgood}  Let $X$ be a completely regular topological space with the following property: there is a sequence of disjoint clopen subsets $\{C_n\}_{n\in \N}$ such that $U=\bigcup_{n\in \N} C_n$ is not closed in $X$ and, for some $x\in \bd U$, every neighbourhood of $x$ meets all but finitely many of the $C_n$. Then, $X$ is not an \rr-good space.  \end{prop}

\begin{proof} The proof of Theorem~\ref{locc} is a model. Define two elements of $\cx$ as follows: $f$ is constantly 1 on $X$, and $g$ is constantly 1 except on $C_n$, $n$ odd, where it is constantly $1+1/n$. It will be shown that $g$ is continuous. Put $U = \bigcup_{n\in\N} C_n$ and $V= \inte (X\setminus U) $. Clearly $g$ is continuous on $U\cup V$. Let $z\in \bd U$. Given $\varep >0$ there is $m\in \N$ such that for all $n>m$, $1/n <\varep$. Then for  $x\in X\setminus \bigcup_{1\le p\le m}C_p$, $|g(x) -g(z)| <\varep$. This shows continuity at all points of $X$. 

It will be shown that $f\wrr g$ does not exist. There are lower bounds. In fact, any function $h$ which is 0 everywhere except on finitely many $C_n$, $n$ even, where it is 1, is continuous and a lower bound. However, since, for each $n\in\N$, $\bd C_n=\vide$, condition $(*)$ of Theorem~\ref{finfg} fails at $x$ (as does condition $(**)$). \end{proof}

 Proposition~\ref{notgood} has many corollaries. Some of them are now listed.
  
 \begin{exs} \label{notrr}  The following are examples of spaces $X$ which are not \rr-good.  (1)~Let $X$ be the one-point compactification of a disjoint union of an infinite collection of compact spaces; e.g., the one-point compactification of an infinite discrete space.  (2)~Let $X$ be any dense subset of $\R$ such that $\R\setminus X$ has an accumulation point in $X$. If $X$ is as stated then $X^n$ is also not \rr-good. (3)~Let $\{F_n\}_{n\in \N}$ be a set of closed subsets of $\R^n$ which are pairwise separated by open sets, $U=\bigcup_{n\in \N}F_n$ is not closed, and for $x\in \bd U$ there is a sequence $\{x_n\}_{n\in \N}$ converging to $x$ with $x_n \in F_n$; then let $X = \ov{\bigcup_{n\in \N}F_n}$ as a subspace of $\R^n$. 
   \end{exs}

\begin{proof} These can be seen to be special cases of Proposition~\ref{notgood}.   \end{proof}

It is next seen that the B-condition is not necessary for a space to be \tbf{rr}-good. 

\begin{ex}\label{notB}  Let $Y$ be an uncountable discrete space and $X= Y\cup \{s\}$ where the neighbourhoods of $s$ are co-countable in $X$. Then, $X$ is an \tbf{rr}-good space which does not satisfy the B-condition.   \end{ex}

\begin{proof} According to \cite[4N]{GJ}, $X$ is a P-space ($\cx$ is regular) and, hence, \tbf{rr}-good. Now partition $Y$ into countable subsets $\{Y_\alpha\}_{\alpha \in A}$ and, for each $\alpha \in A$, let $e_\alpha$ be the idempotent whose support is $Y_\alpha$. Each $Y_\alpha$ is thus a cozero set in $X$ with empty boundary. However, $\bigcup_AY_\alpha = Y$ whose boundary is $\{s\}$. Hence, the B-condition fails. \end{proof}

Recall that for a completely regular space $X$, $C^*(X)$ is the ring of bounded real-valued continuous functions on $X$. 
\begin{prop}\label{c*x}  Let $X$ be a completely regular topological space.  Then $X$ is  \rr-good  if and only if $\text{C\,}^*(X)$ is an \rr-good ring.    \end{prop}

\begin{proof}  If $X$ is \rr-good and $f, g\in \cxs$ and $h = f\wrr g$, calculated in $\cx$, then $h\in \cxs$. 

In the other direction, let $f,g\in \cx$. Put $f_b=f/(1+f^2)$ and $g_b= g/(1+g^2)$.  Then $f_b, g_b\in \cxs$, as are $f_{bb}= f_b/(1+g^2)$ and $g_{bb}=g_b/(1+f^2)$. Let $h= f_{bb} \wrr g_{bb}$. Then $h^2= hf_{bb} = hg_{bb}$.  Hence, $h^2(1+f^2)^2(1+g^2)^2 = h(1+f^2)(1+g^2)f = h(1+f^2)(1+g^2)g$ showing that $h(1+f^2)(1+g^2) $ is a lower bound for $f$ and $g$. Let $k\in \cx$ be a lower bound for $f$ and $g$. Then, $k^2= kf=kg$ and, hence, $k/(1+f^2)(1+g^2)$ is a lower bound for $f_{bb}$ and $g_{bb}$.  It follows that $k/(1+f^2)(1+g^2)\in \cxs$ and $k/(1+f^2)(1+g^2) \rro h$. Then, by Proposition~\ref{lsl}~(i), $k\rro h(1+f^2)(1+g^2)$ showing that $h(1+f^2)(1+g^2) = f\wrr g$ in $\cx$. \end{proof}

Since for any space $X$, $C^*(X)\cong C(\beta X)$, where $\beta X$ is the \v{C}ech-Stone compactification of $X$, Proposition~\ref{c*x} shows that if $X$ is \rr-good so is $\beta X$. It is not hard to show that if $X$ is \rr-good so is any space $V$, $X\sbq V\sbq \beta X$. Moreover, since for any space $Y$ between $X$ and its realcompactification $\upsilon X$, $C(Y) = \cx$, if $X$ is \rr-good so is $V$. (See \cite[Chapter~8]{GJ}.)

It is next shown that cozero sets in \rr-good spaces are also \rr-good, for example, any open set in an \rr-good metric space.

\begin{prop} \label{cozgood}  Let $X$ be a completely regular topological space which is \rr-good. Let $Y = \coz k$ where $k\in \cx$. 
The space $Y$ is \rr-good.

\end{prop}

\begin{proof} First note that $k$ may be assumed to be bounded. If $f\in \text{C}^*(V)$ then $fk$ can be extended to $\tilde{f}\in \cx$ by defining $\tilde{f}(x)=0$ if $x\in X\setminus Y$ (see \cite[Proposition~1.1]{BH}). Notice also that an \rr-lower bound of two bounded functions is necessarily bounded. 

It will first be shown that $\text{C}^*(Y)$ is an \rr-lower semi-lattice; the result for $\text{C}(Y)$ follows from Proposition~\ref{c*x}. Consider $f,g\in \text{C}^*(V)$. Extend $fk$ and $gk$ as above to get $\tilde{f}$ and $\tilde{g}$ in $\cx$. Let $H = \tilde{f}\wrr \tilde{g}$. Then, in particular, $H^2= \tilde{f}H= \tilde{g}H$. Set $h= H\res _V$ showing that $h^2= fkh= gkh$. Consider $1/k\in \text{C}(Y)$. Then, $(h/k)^2 = f(h/k) = g(h/k)$. Thus $h/k$ is a lower bound for $f$ and $g$. Since $f,g\in \text{C}^*(V)$, it follows that $h/k\in \text{C}^*(V)$. 

Suppose next  that $a$ is also a lower bound for $f$ and $g$. Then, $ak \rro fk $ and $ak\rro gk$. Consider $\tilde{a}$. It follows that $(\tilde{a})^2 = \tilde{a}\tilde{f} = \tilde{a}\tilde{g}$ since the equality holds on $Y$ and all the functions are 0 on $X\setminus Y$. Therefore, $\tilde{a}\le H$ and $\tilde{a}\res_Y$ is a lower bound of $fk$ and $gk$, showing that $ak\rro h$. Then, upon multiplication by $1/k$, $a\le h/k$. Thus, $h/k = f\wrr g$. 
\end{proof}

It is not true that the restriction homomorphism $\cx \to \text{C}(Y)$ in Proposition~\ref{cozgood} preserves $\wrr$. 

\begin{ex} \label{cozbad} Let $X=\R$ and $Y = (0,1)$, a cozero set, and $\rho\colon \cx\to \text{C}(Y)$. Then, $\rho$ does not preserve \rr-infima.\end{ex}

\begin{proof} Consider two function $f$ and $g$ in $\cx$ which are never 0 except at $x=1/2$ and only coincide on $[0,1/2]$. Then, $f\wrr g=0$ in $\cx$ but $\rho (f)\wrr \rho (g) = h\ne 0$ where $h$ coincides with $f$ and $g$ on $(0,1/2)$ and is 0 on $[1/2,1)$. \end{proof}

In the context of $\cx$ it is possible to construct other examples of ring surjections where countable \rr-orthogonal sets lift to \rr-orthogonal sets.  (See also Theorem~\ref{liftwb}.)  A subspace $V$ of $X$ is C-embedded if every element of $\text{C}(V)$ extends to an element of $\cx$, i.e., the restriction map $\rho\colon \cx \to \text{C}(V)$ is surjective. 

\begin{prop}\label{liftcx} Let $X$ be an \rr-good space, $V$ a C-embedded subspace and $\rho\colon \cx \to \text{C}(V)$ the restriction homomorphism.  If $\{f_n\}_{n\in \N}$ is an \rr-orthogonal set in $\text{C}(V)$ then there is an \rr-orthogonal set $\{F_n\}_{n\in \N}$ in $\cx$ such that for all $n\in \N$, $\rho(F_n) =f_n$. \end{prop}

\begin{proof} The proof will be by induction. For $F_1$, pick any pre-image of $f_1$.  Let $n\ge 1$ and suppose that an \rr-orthogonal set $F_1,\ldots, F_n$ has been found with, for $1\le i\le n$, $\rho(F_i) = f_i$. Choose $G\in \cx$ to be any pre-image of $f_{n+1}$. 

For $1\le i\le n$, let $U_i=\coz (G\wrr F_i)$. If all the $U_i$ are empty then put $G = F_{n+1}$. In any  case, notice that $U_i\cap V=\vide$ since $G\res_V\wrr F_i\res_V =0$ for $1\le i\le n$.

It will be shown that, in any case, the $U_i$ are pairwise disjoint.  Now suppose that for some $1\le i< j\le n$ that $U_i\cap U_j\ne \vide$.  On $U_i\cap U_j$, all of $G, F_i$ and $F_j$ coincide.  Suppose that  $\bd(U_i\cap U_j)\ne \vide $ and that all three functions are 0 on it. Then let $H$ be such that $H$ coincides with $G, F_i$ and $F_j$ on $U_i\cap U_j$ but is 0 elsewhere. This function is continuous. Moreover, $H\rro F_i$ and $H\rro F_j$.  This is impossible since $H\ne 0$ and $F_i\wrr F_j=0$. If follows that for some $x_0\in \bd (U_i\cap U_j)$ that $G(x_0) = F_i(x_0)= F_j(x_0) \ne 0$ (they coincide by continuity). In this case $x_0\in U_i\cap U_j$.  To see this consider a net $\{\lambda_\alpha\}_{\alpha \in A}$ in $U_i\cap U_j$ which converges to $x_0$. Think of this as a net in $U_i$; since $F_i(x_0)\ne 0$, the limit of the net, $x_0$, is in $U_i$ and not in $\bd U_i$. Similarly, $x_0$ is in $U_j$. Thus, $x_0\in U_i\cap U_j$. This is impossible by the definition of the boundary.  The only remaining possibility is that $\bd (U_i \cap U_j) =\vide$. This means that $U_i\cap U_j$ is clopen and thus the cozero set of an idempotent, say $e\in \cx$. But then $eG$ will be a lower bound for $F_i$ and for $F_j$. This cannot happen unless $e=0$ and then $U_i\cap U_j= \vide$. 

Hence, for each $1\le i< j\le n$, $U_i\cap U_j=\vide$. Notice also that each $U_i\sbq X\setminus V$. Now $G$ can be modified as follows: Define $G'$ by

\begin{equation*}
G'(x) =
\begin{cases}
0, & \text{if}\; x\in \bigcup_{i=1}^n\ov{U_i}\\
G(x), & \text{otherwise}
\end{cases}
\end{equation*}
The function $G$ is zero on each $\bd U_i$ which shows that $G'\in \cx$. Since $G$ and $G'$ coincide on $V$, they have the same image in $\text{C}(V)$. It remains to be shown that $F_i$ and $G'$ are \rr-orthogonal for $i=1, \ldots n$. Suppose, for some $i$, that $H$ is an \rr-lower bound of $F_i$ and of $G'$. Put $W = X\setminus \bigcup_{j=1}^n\ov{U_j}$. Since $G'$ is 0 on $ \bigcup_{j=1}^n\ov{U_i}$, $\coz H\sbq W$. Recall that $\coz (F_i\wrr G) = U_i$. Define a new function $K$ as follows:
\begin{equation*}
K(x) =
\begin{cases}
(F_i\wrr G)(x), & x\in U_i\\
H(x), &x\in \coz H\\
0, &\text{otherwise}
 \end{cases}
\end{equation*}
The function $K$ is continuous and is an \rr-lower bound of $F_i$ and of $G$. However, $F_i\wrr G\rro K$ showing that $K = F_i\wrr G$. But then, $\coz H=\vide$ and $H=0$.

  \end{proof}
  
  As examples where Proposition~\ref{liftcx} can be used, recall that any compact subspace $V$ of $X$ is C-embedded (\cite[6J]{GJ}) and, hence, if $X$ is \rr-good the proposition applies. It is also easy to see that if, in the above, the subspace $V$ is only C$^*$-embedded then countable \rr-orthogonal sets of bounded elements of $\text{C}(V)$ lift to \rr-orthogonal sets in $\cx$.

\bibliographystyle{amsplain}

\begin{thebibliography}{12}

\bibitem{AHS} J. Ad\'amek, H. Herrlich and G.E. Strecker, Abstract and Concrete Categories. Springer Verlag, Online Edition, 2004.
\bibitem{AR} V.A. Andrunakievi\v{c} and J.M. Rjabuhin, \emph{Rings without nilpotent elements and completely simple ideals}, Soviet Math.\ Dokl.\ \tbf{9} (1968), 565--568.
\bibitem{B} G. M. Bergman, \emph{Hereditary commutative rings and centres of hereditary rings}, Proc.\ London Math.\ Soc. \tbf{23} (1971), 214--236.
\bibitem{BPR} G.F. Birkenmeier, J.K. Park and S.T. Rizvi, Extensions of Rings and Modules. Birkh\"auser/Springer, 2013.
\bibitem{BH} R.L. Blair and A.W. Hager, \emph{Extensions of zero-sets and of real-valued functions}, Math.\ Zeit.\ \tbf{136} (1974) 41--52.
\bibitem{BR} W.D. Burgess and R. Raphael, \emph{Abian's order and orthogonal completions for reduced rings}, Pacific J. Math.\ \tbf{54} (1974), 56--74.
\bibitem{E} R. Engelking, General Topology, Heldermann, Berlin, 1989.
\bibitem{FU} C. Faith and Y. Utumi, \emph{Intrinsic extensions of rings}, Pacific J. Math.\ \tbf{14} (1964), 505--512.
\bibitem{GJ} L. Gillman and J. Jerison, Rings of Continuous Functions. Graduate Texts in Mathematics \tbf{43}, Springer, 1976.

\bibitem{H} M. Hochster, \emph{Prime ideal structure in commutative rings}, Trans.\ Amer.\ Math.\ Soc.\ \tbf{142} (1969), 43--60.
\bibitem{J} P.T. Johnstone, Stone Spaces, Cambridge studies in advanced mathematics, vol.\ 3, Cambridge University Press, 1992.

\bibitem{K} J.L. Kelley, General topology. Graduate Texts in Mathematics \tbf{27}, Springer-Verlag, 1975.

\bibitem{L}T.Y. Lam, A First Course in Noncommutative Rings. Graduate Texts in Mathematics \tbf{131}, Springer, 2001.
\bibitem{NR} S.B. Niefield and K.I. Rosenthal, \emph{Sheaves of integral domains on Stone spaces}, J. Pure Appl.\ Algebra \tbf{47} (1987), 173--179.
\bibitem{R} R. Raphael, \emph{On the countable orthogonal lifting property}, in preparation.
\bibitem{St} P.N. Stewart, \emph{Semi-simple radical classes}, Pacific J. Math.\ \tbf{32} (1970), 249--255.
\bibitem{S} H.H. Storrer, \emph{Epimorphismen von kommutativen Ringen}, Comment.\ Math.\ Helv.\ \tbf{43} (1968), 378--401.

\bibitem{Wi} R. Wiegand, \emph{Modules over universal regular rings}, Pacific J.\ Math. \tbf{39} (1971), 807--819.
\bibitem{W} S. Willard, General Topology, Addison-Wesley, 1970.

\end{thebibliography}

\end{document}